\documentclass[a4paper, 11pt, reqno]{amsart}
\usepackage{amsfonts, amssymb}
\usepackage{color}
\usepackage{cases}
\usepackage[varg]{txfonts}

\title[Moment maps and Isoparametric hypersurfaces]{Moment maps and Isoparametric hypersurfaces in spheres --- Hermitian cases}
\date{\today}
\author{Shinobu Fujii}
\address{General Education Division, Oshima College of Maritime Technology, 
1091-1~Komatsu, SuouOshima-cho, Oshima-gun, Yamaguchi, 742-2193, JAPAN}
\email{shinofu@oshima-k.ac.jp}
\author{Hiroshi Tamaru}
\address{Department of Mathematics, Hiroshima University, 1-3-1~Kagamiyama, 
Higashi-Hiroshima, 739-8526, JAPAN}
\email{tamaru@math.sci.hiroshima-u.ac.jp}
\subjclass[2010]{Primary~53C40, Secondary~53D20}
\keywords{Isoparametric hypersurfaces; Cartan-M\"unzner polynomials; Hermitian symmetric spaces; moment maps}
\numberwithin{equation}{section}

\def\g{\mathfrak{g}}
\def\k{\mathfrak{k}}
\def\p{\mathfrak{p}}
\def\a{\mathfrak{a}}
\def\h{\mathfrak{h}}
\def\bR{\mathbb{R}}

\def\0{\mathbf{0}}
\DeclareMathOperator{\grad}{grad}

\DeclareMathOperator{\SO}{SO}
\DeclareMathOperator{\U}{U}
\DeclareMathOperator{\Spin}{Spin}

\DeclareMathOperator{\gu}{\mathfrak{u}}
\DeclareMathOperator{\SU}{SU}
\DeclareMathOperator{\Sp}{Sp}

\DeclareMathOperator{\ad}{ad}
\DeclareMathOperator{\Tr}{Tr}
\DeclareMathOperator{\rank}{rank}

\DeclareMathOperator{\id}{id}
\providecommand{\ideal}[1]{\left\langle #1 \right\rangle}
\providecommand{\Norm}[1]{\left\lVert #1 \right\rVert}

\newtheorem{Thm}{Theorem}[section]
\newtheorem{Lem}[Thm]{Lemma}

\newtheorem{Prop}[Thm]{Proposition}
\newtheorem*{MainThm}{Main Theorem}
\theoremstyle{definition}
\newtheorem{Dfn}[Thm]{Definition}
\theoremstyle{remark}
\newtheorem{Rmk}[Thm]{Remark}

\begin{document}
\begin{abstract}
We are studying a relationship between isoparametric hypersurfaces in spheres with four distinct principal curvatures and the moment maps of certain Hamiltonian actions. 
In this paper, we consider the isoparametric hypersurfaces obtained from the isotropy representations of compact irreducible Hermitian symmetric spaces of rank two. 
We prove that the Cartan-M\"unzner polynomials of these hypersurfaces can be written as squared-norms of the moment maps for some Hamiltonian actions. 
The proof is based on the structure theory of symmetric spaces. 
\end{abstract}

\maketitle

%
%
\section{Introduction}

A hypersurface $N$ in a Riemannian manifold $(M, \ideal{\ ,\ })$ is called an \textit{isoparametric hypersurface} if $N$ is a level set of an isoparametric function on $M$.
Here, a smooth function $\varphi$ on $M$ is said to be \textit{isoparametric} if there exist two smooth functions $A(t)$ and $B(t)$ on $\bR$ which satisfy 
\begin{equation}
\left\{\begin{tabular}{c}
$\Norm{\grad \varphi}^2 = \ideal{\grad \varphi, \grad \varphi} = A\circ \varphi$,\\[5pt]
\hspace*{-22mm}$\Delta \varphi = B\circ \varphi$.
\end{tabular}\right.
\end{equation}
We refer to Cecil \cite{Cecil}, Thorbergsson \cite{Thorbergsson}, and the references therein, for history and general theory of isoparametric hypersurfaces. 
Note that isoparametric hypersurfaces in the Euclidean spaces $\mathbb{R}^n$ and the hyperbolic spaces $\mathbb{H}^n$ have been classified completely. 

In this paper, we consider isoparametric hypersurfaces in the spheres $S^n$. 
It is known that a hypersurface in $S^n$ is isoparametric if and only if it has constant principal curvatures. 
Hence, a homogeneous hypersurface in $S^n$ is isoparametric. 
Homogeneous hypersurfaces in spheres have been classified completely by Hsiang and Lawson \cite{Hsiang--Lawson}, and can be characterized as principal orbits of the isotropy representations of symmetric spaces of rank two. 
Let $g$ denote the number of distinct principal curvatures of an isoparametric hypersurface in $S^n$. 
Then, it is known by M\"unzner \cite{Munzner2} that
\begin{align}
g \in \{ 1 , 2 , 3 , 4 , 6 \} . 
\end{align}
In the case of $g \neq 4$, isoparametric hypersurfaces in $S^n$ have been classified completely (for $g \leq 3$ case by E.~Cartan, and for $g=6$ case by Dorfmeister and Neher \cite{DN}, and Miyaoka \cite{Miyaoka6}). 
In fact, all of them are homogeneous for $g \neq 4$. 
In contrast, there exist infinitely many non-homogeneous isoparametric hypersurfaces in the case of $g = 4$. 
Such examples have been constructed by Ozeki and Takeuchi \cite{Ozeki--Takeuchi1}, and Ferus, Karcher and M\"unzner \cite{FKM}. 
On the classification problem of isoparametric hypersurfaces in spheres with $g=4$, very big progress has recently been made by Cecil, Chi and Jensen \cite{CCJ}, Immervoll \cite{Immervoll} and Chi \cite{Chi2008}, \cite{Chi2011b}, \cite{Chi2011a}, but a complete classification seems to be still open. 

In the study of isoparametric hypersurfaces in spheres, Cartan-M\"unzner polynomials play important roles. 
We here briefly recall some results by M\"unzner \cite{Munzner, Munzner2}. 
Let $M$ be an isoparametric hypersurface in $S^n$ with $g$ distinct principal curvatures. 
We denote its principal curvatures by $\lambda_1 < \dots < \lambda_g$, and their multiplicities by $m_1,\dots, m_g$, respectively. 
Then, $M$ is a level set of $\varphi |_{S^n}$, where $\varphi : \bR^{n + 1} \to \bR$ is a homogeneous polynomial function of degree $g$ which satisfies 
\begin{equation}\label{eq:Munzner}
\left\{\begin{tabular}{l}
$\Norm{\grad \varphi(P)}^2 = g^2\Norm{P}^{2 g - 2}$, \\[5pt]
\hspace*{11.25mm}$\Delta \varphi(P) = \dfrac{m_2 - m_1}{2}g^2\Norm{P}^{g - 2}$. 
\end{tabular}\right.
\end{equation}
A homogeneous polynomial satisfying \eqref{eq:Munzner} is called a \textit{Cartan-M\"unzner polynomial} of degree $g$. 
Conversely, if $\varphi : \bR^{n + 1} \to \bR$ is a Cartan-M\"unzner polynomial of degree $g$, then every regular level set of $\varphi |_{S^n}$ is an isoparametric hypersurface in $S^n$ with $g$ distinct principal curvatures. 
Note that $m_1$ and $m_2$ determine all multiplicities, 
from the fact that 
\begin{equation}
m_i = m_{i + 2} \ \ \text{for $i \bmod g$} . 
\end{equation}

We are interested in isoparametric hypersurfaces in spheres with four distinct principal curvatures. 
Our expectation is that every such hypersurface is related to a moment map for a certain Hamiltonian action. 
More precisely, we expect that every Cartan-M\"unzner polynomial of degree four can be described as a squared-norm of a moment map. 
This idea was first proposed by the first author \cite{Fujii}. 
If our expectation is true, then it will give a new view point for the study of isoparametric hypersurfaces in spheres. 
Furthermore, it would potentially be useful for the (re)classification of isoparametric hypersurfaces in spheres with $g=4$. 
We note that Miyaoka \cite{Miyaoka} recently showed that all known Cartan-M\"unzner polynomials of degree four can be expressed in terms of a moment map. 
Thus, her result supports our expectation, but the formulation is different from ours. 

In this paper, we show that our expectation is true for ``Hermitian case'', that is, the homogeneous isoparametric hypersurfaces obtained from the isotropy representations of compact irreducible Hermitian symmetric spaces of rank two. 
To state our main result, we set up some notations. 
Let $G / K$ be a compact irreducible Hermitian symmetric space of arbitrary rank. 
We denote the Lie algebras of $G$ and $K$ by $\g$ and $\k$ respectively. 
Let $\g = \k \oplus \p$ be the Cartan decomposition of $\g$. 
The $K$-action on $\p$ by the adjoint action is called the isotropy representation of $G/K$. 
We define the $K$-invariant inner product $\ideal{\ ,\ }$ on $\g$ by
\begin{equation}
\label{eq:naiseki}
\ideal{X, Y} := -B(X, Y)\ \ \text{for}\ \ X, Y \in \g, 
\end{equation}
where $B$ is the Killing form of $\g$. 
We also denote by $\ideal{\ ,\ }$ its restriction to $\p$ and to $\k$ respectively. 
Since $G / K$ is irreducible and Hermitian, there exists the unique $Z$ (up to sign) in the center $C(\k)$ of $\k$ such that 
\begin{equation}
\label{eq:J}
J := \left. \ad_Z \right|_{\p}
\end{equation}
gives a complex structure on $\p$. 
We can identify $T_P \p$ with $\p$ because $\p$ is a vector space. 
Thus, we can define a canonical symplectic form $\omega$ on $\p$ by
\begin{equation}
\omega_P( X_P, Y_P) := \ideal{J(X_P), Y_P}\ \ \text{for}\ \ X, Y \in \mathfrak{X}(\p)\ \ \text{and}\ \ P \in \p, 
\end{equation}
where $\mathfrak{X}(\p)$ is the set of all vector fields on $\p$. 
Then, $(\p, \omega)$ is a symplectic manifold. 
As we review in Section~\ref{sect:Preliminarily}, the isotropy representation of $G / K$ is a Hamiltonian $K$-action on $(\p, \omega)$. 
We denote its moment map by $\mu : \p \to \k^\ast$. 
Our main result is the following:

\begin{MainThm}
Let $G / K$ be a compact irreducible Hermitian symmetric space of rank two, and $\mu$ be the moment map for the isotropy representation of $G / K$ as above. 
Then, there exists a $K$-invariant norm on $\k^\ast$ such that the squared-norm of $\mu$ is a Cartan-M\"unzner polynomial of degree four. 
\end{MainThm}

Recall that homogeneous hypersurfaces in spheres with $g=4$ are precisely principal orbits of the isotropy representations of the following compact irreducible symmetric spaces: 
\begin{enumerate}
\item $\SO(2 + n) / \SO(2) \times \SO(n)$, 
\item $\SU(2 + n) / \mathrm{S}(\U(2) \times \U(n))$, 
\item $\SO(10) / \U(5)$, 
\item $\mathrm{E}_6 / \U(1) \times \Spin(10)$, 
\item $\Sp(2 + n) / \Sp(2) \times \Sp(n)$, 
\item $\SO(5) \times \SO(5) / \SO(5)$.
\end{enumerate}
Among them, (1), (2), (3) and (4) are Hermitian. 
For classical ones (1), (2) and (3), our main theorem has already been proved by the first author \cite{Fujii}. 
Note that the proofs in \cite{Fujii} use matrices expressions, and the arguments are case-by-case. 
In this paper, we prove our main theorem in terms of the structure theory of symmetric spaces, and hence applicable to both classical cases (1), (2), (3) and exceptional one (4), simultaneously. 
Furthermore, our arguments do not depend on the classification of homogeneous hypersurfaces in spheres. 

This article is organized as follows. 
In Section~\ref{sect:Preliminarily}, we study the function $f_{a,b}$, which is the squared-norm of the moment map $\mu$ 
with respect to the $K$-invariant norm $|| \cdot ||_{a,b}$. 
In Section~\ref{sect:arbitrary rank}, we give formulas for the squared-norm of the gradient and the Laplacian of $f_{a,b}$, where $G/K$ is of arbitrary rank. 
Section~\ref{sect:rank two} simplifies these formulas by assuming that $G/K$ is of rank two. 
In Section~\ref{sect:Main Theorem}, we prove that there exist $a, b \in \bR$ such that $f_{a,b}$ is a Cartan-M\"unzner polynomial of degree four, which concludes the main theorem. 

%
%
\section{Preliminarily}
\label{sect:Preliminarily}

Let $G/K$ be a compact irreducible Hermitian symmetric space. 
In this section, we define and study the functions $f_{a, b}(P)$, 
the squared-norms of the moment map $\mu$ with respect to certain invariant norms $|| \cdot ||_{a,b}$. 

First of all, we recall the formula for the moment map $\mu$. 
Note that $\ideal{\ ,\ }$ is the $K$-invariant inner product on $\k$ defined in \eqref{eq:naiseki}, and $Z$ is an element in the center $C(\k)$ of $\k$ which defines a complex structure on $\p$. 

\begin{Prop}[Fujii {\cite[Proposition 2.2]{Fujii}}, see also \cite{Ohnita}]
\label{Prop:2.1}
The isotropy representation of $G/K$ is a Hamiltonian $K$-action on $(\p, \omega)$. 
Under the identification $\k^\ast \cong \k$ by $\ideal{\ ,\ }$, the moment map $\mu : \p \to \k$ is given by 
\begin{equation}
\mu(P) = \dfrac{1}{2}[P, [P, Z]] \ \ \ \text{for} \ \ P \in \p . 
\end{equation}
\end{Prop}

Next, we take $K$-invariant norms on $\k$ and the squared-norms of $\mu$. 
Let $|| \cdot ||$ be the canonical norm on $\g$, that is, 
\begin{equation}
|| X ||^2 := \ideal{X,X} \ \ \ \text{for} \ \ X \in \g . 
\end{equation}
Its restriction to $\k$ is $K$-invariant, but we have more $K$-invariant norms on $\k$. 
Since $G / K$ is Hermitian, one has a decomposition $\k = \gu(1) \oplus \k'$ as a Lie algebra, where $\gu(1) = \bR Z$. 
Let 
\begin{equation}
\pi_1 : \k \to \gu(1) , \quad \pi_2 : \k \to \k' 
\end{equation}
denote the canonical projections of $\k$ onto $\gu(1)$ and $\k'$, respectively. 
Hence, for $a,b \in \bR$, we have a $K$-invariant norm $\Norm{ \, \cdot \, }_{a,b}$ on $\k$ by 
\begin{equation}
\Norm{X}_{a,b}^2 := a \Norm{\pi_1(X)}^2 + b \Norm{\pi_2(X)}^2 
\ \ \text{for} \ \ X \in \k . 
\end{equation}

\begin{Dfn}
The \textit{weighted squared-norm} $f_{a,b} : \p \to \bR$ 
of the moment map $\mu$ is defined by 
\begin{equation} 
f_{a, b}(P) := \Norm{\mu(P)}_{a,b}^2 \ \ \text{for} \ \ P \in \p . 
\end{equation} 
\end{Dfn} 

Note that the weighted squared-norm $f_{a,b}$ has already been introduced by the first author \cite{Fujii}. 
Since $\mu$ is $K$-equivariant and $\Norm{\ \cdot\ }_{a,b}$ is $K$-invariant, one can easily see the following. 

\begin{Prop}
The weighted squared-norm $f_{a, b}$ is a $K$-invariant function on $\p$. 
\end{Prop}

In the remaining of this section, we give a formula for $f_{a,b}(P)$. 
It was obtained in \cite[Sections 2 and 3]{Fujii} that 
\begin{equation}
\label{eq:2.5}
f_{a, b}(P) = b\Norm{\mu(P)}^2 + \dfrac{a - b}{4\Norm{Z}^2}\Norm{P}^4. 
\end{equation}
To rewrite this formula, we calculate $\Norm{Z}^2$. 
Denote by $N := \dim \p$. 

\begin{Lem}\label{Lem:3.1}
We have $\Norm{Z}^2 = N$. 
\end{Lem}

\begin{proof}
By the definition of our inner product, we have
\begin{equation}\label{eq:3.1}
\Norm{Z}^2 = \ideal{Z, Z} = - B(Z,Z) = - \Tr\left( \left( \ad_Z \right)^2 \right).  
\end{equation}
Since $(\ad_Z)^2$ preserves $\k$ and $\p$ respectively, one has 
\begin{equation}
- \Tr \left( \left( \ad_Z \right)^2 \right) 
= - \Tr \left( \left. (\ad_Z)^2 \right|_{\k} \right) 
- \Tr \left( \left. (\ad_Z)^2 \right|_{\p} \right) . 
\end{equation}
Because $Z \in C(\k)$, we find $(\ad_Z)^2 |_{\k} = 0$. 
On the other hand, since $\left.\ad_Z\right|_{\p}$ is a complex structure on $\p$, it follows that $(\ad_Z)^2 |_{\p} = -\id_\p$. 
We thus have 
\begin{equation}\label{eq:3.2}
- \Tr \left( \left( \ad_Z \right)^2 \right) 
= - \Tr \left( -\id_\p \right) = \dim \p = N . 
\end{equation}
This completes the proof. 
\end{proof}

\begin{Prop}
\label{Prop:f-ab}
We have
\begin{equation}\label{eq:3.3}
f_{a, b}(P) = b \Norm{\mu(P)}^2 + \dfrac{a - b}{4 N}\Norm{P}^4. 
\end{equation}
\end{Prop} 

\begin{proof}
It follows easily from \eqref{eq:2.5} and Lemma~\ref{Lem:3.1}. 
\end{proof}

Note that the results stated in this section are independent of the rank of $G / K$, but it is essential that $G/K$ is Hermitian. 

%
%
\section{The gradient and the Laplacian in the case of arbitrary rank}\label{sect:arbitrary rank}

Let $f_{a, b}(P)$ be the weighted squared-norm of the moment map $\mu$. 
In this section, we compute the squared-norm of the gradient and the Laplacian of $f_{a, b}(P)$, 
where the rank of $G/K$ is arbitrary. 
Throughout this section, we denote by $r := \rank G/K$ and $N := \dim \p$. 
We fix an orthonormal basis $\{ P_i \}$ of $\p$. 

%
%

\subsection{The squared-norm of the gradient}

In this subsection, we compute the squared-norm of the gradient of $f_{a, b}(P)$. 
Recall that the gradient is given by 
\begin{equation}
\grad f_{a, b}(P) = \sum \dfrac{\partial f_{a, b}}{\partial P_i} P_i . 
\end{equation}
Recall that $f_{a, b}(P)$ can be written as a linear combination of $\Norm{P}^4$ and $\Norm{\mu(P)}^2$. 
Hence we start from calculating their partial differentials. 
We use the complex structure $J = \ad_Z |_{\p}$ on $\p$ defined in \eqref{eq:J}. 

\begin{Lem}\label{Lem:3.4}
Partial differentials of $\Norm{P}^4$ and $\Norm{\mu(P)}^2$ of order one satisfy 
\begin{enumerate}
\item $(\partial \Norm{P}^4) / (\partial P_i) = 4\Norm{P}^2\ideal{P, P_i}$, 
\item $(\partial \Norm{\mu(P)}^2) / (\partial P_i) = 2\ideal{[\mu(P), J(P)], P_i}$. 
\end{enumerate}
\end{Lem}

\begin{proof}
Recall that, for each $P_i$, the partial differential of a function $f(P)$ of order one is given by
\begin{equation}\label{eq:3.18}
\dfrac{\partial f(P)}{\partial P_i} = \lim_{t \to 0} \dfrac{f(P + t P_i) - f(P)}{t} . 
\end{equation}

We show (1). 
First of all, it is easy to see that 
\begin{equation}\label{eq:3.19}
\dfrac{\partial P}{\partial P_i} = \lim_{t \to 0} \dfrac{(P + t P_i) - P}{t} = P_i . 
\end{equation}
One can also see that 
\begin{equation}\label{eq:3.20}
\dfrac{\partial \Norm{P}^2}{\partial P_i} 
= \dfrac{\partial \ideal{P, P}}{\partial P_i} 
= 2 \ideal{P, \dfrac{\partial P}{\partial P_i}} 
= 2 \ideal{P, P_i} . 
\end{equation}
Therefore, we conclude (1) by 
\begin{equation}\label{eq:3.21}
\dfrac{\partial \Norm{P}^4}{\partial P_i} 
= \dfrac{\partial (\Norm{P}^2)^2}{\partial P_i} 
= 2\Norm{P}^2\dfrac{\partial \Norm{P}^2}{\partial P_i} 
= 4\Norm{P}^2\ideal{P, P_i}.  
\end{equation}

Next, we prove (2). 
Since $\mu(P) = (1/2) [P,[P,Z]]$ by Proposition~\ref{Prop:2.1}, we obtain
\begin{equation}
\dfrac{\partial \mu(P)}{\partial P_i} 
= \dfrac{1}{2} ([\dfrac{\partial P}{\partial P_i}, [P, Z]] + [P, [\dfrac{\partial P}{\partial P_i}, Z]]) 
= \dfrac{1}{2} 
([P_i, [P, Z]] + [P, [P_i, Z]]) . 
\end{equation}
It follows from the Jacobi identity and $Z \in C(\k)$ that
\begin{equation}
[P, [P_i, Z]] 
= - [P_i, [Z, P]] - [Z , [P, P_i]] 
= [P_i, [P, Z]] 
= [J(P), P_i] . 
\end{equation}
This concludes
\begin{equation}\label{eq:3.23}
\dfrac{\partial \mu(P)}{\partial P_i} 
= [J(P), P_i] . 
\end{equation}
Thus, we get
\begin{equation}\label{eq:3.25}
\dfrac{\partial \Norm{\mu(P)}^2}{\partial P_i} 
= 2\ideal{\mu(P), \dfrac{\partial \mu(P)}{\partial P_i}} 
= 2\ideal{\mu(P), [J(P), P_i]} . 
\end{equation}
Since $\ad_{J(P)}$ is skew-symmetric, this completes the proof of (2). 
\end{proof}

From the above calculations of differentials of order one, we can compute the gradients of $\Norm{P}^4$ and $\Norm{\mu(P)}^2$. 

\begin{Lem}
\label{lem:gradient}
We have 
\begin{enumerate}
\item 
$\grad \Norm{P}^4 = 4 \Norm{P}^2 P$, 
\item 
$\grad \Norm{\mu(P)}^2 = 2 [\mu(P), J(P)]$. 
\end{enumerate}
\end{Lem}

\begin{proof}
The claim (1) follows easily from Lemma \ref{Lem:3.4} (1), since 
\begin{equation}\label{eq:3.29}
\grad \Norm{P}^4 
= \sum_{i = 1}^{N}\dfrac{\partial \Norm{P}^4}{\partial P_i}P_i 
= \sum_{i = 1}^{N}4 \Norm{P}^2\ideal{P, P_i}P_i 
= 4 \Norm{P}^2 P .
\end{equation}
We show (2). 
From Lemma \ref{Lem:3.4} (2), it follows that
\begin{equation}\label{eq:3.28}
\grad \Norm{\mu(P)}^2 
= \sum_{i = 1}^{N} \dfrac{\partial \Norm{\mu(P)}^2}{\partial P_i}P_i 
= 2\sum_{i = 1}^{N} \ideal{ [\mu(P), J(P)], P_i }P_i . 
\end{equation}
Note that $[\mu(P), J(P)] \in \p$, because $\mu(P) \in \k$ and $J(P) \in \p$. 
Hence, this concludes the proof of (2). 
\end{proof}

\begin{Rmk}
Although we computed $\grad \Norm{\mu(P)}^2$ directly, it is well-known in symplectic geometry (we refer to Kirwan \cite[Lemma 6.6]{Kirwan}). 
\end{Rmk}

Now it is easy to compute the gradient of $f_{a, b}(P)$. 

\begin{Lem}\label{Lem:3.5}
We have
\begin{equation}\label{eq:3.26}
\grad f_{a, b}(P) = 2 b [\mu(P), J(P)] + \dfrac{(a - b)\Norm{P}^2}{N}P. 
\end{equation}
\end{Lem}

\begin{proof}
We use the expression of $f_{a, b}(P)$ given in Proposition \ref{Prop:f-ab}. 
By linearity of $\grad$, we get 
\begin{equation}\label{eq:3.27}
\grad f_{a, b}(P) = b\grad \Norm{\mu(P)}^2 + \dfrac{a - b}{4 N}\grad \Norm{P}^4. 
\end{equation}
Therefore, Lemma \ref{lem:gradient} easily concludes the proof. 
\end{proof}

We are now in position to compute the squared-norm of the gradient of $f_{a, b}(P)$. 

\begin{Prop}\label{Prop:3.7}
We have
\begin{equation}\label{eq:3.30}
\Norm{\grad f_{a, b}(P)}^2 = 4 b^2 \Norm{[J(P), \mu(P)]}^2 + \dfrac{8 b (a - b)}{N}\Norm{P}^2\Norm{\mu(P)}^2 + \dfrac{(a - b)^2}{N^2}\Norm{P}^6.  
\end{equation}
\end{Prop}

\begin{proof}
One knows $\grad f_{a,b}(P)$ from Lemma~\ref{Lem:3.5}. 
We thus get
\begin{equation}\label{eq:3.31}
\begin{split}
\Norm{\grad f_{a, b}(P)}^2 
&= 4 b^2 \Norm{[\mu(P), J(P)]}^2 + \dfrac{4 b (a - b)\Norm{P}^2}{N}\ideal{[\mu(P), J(P)], P}\\
&\hspace{4cm}+ \dfrac{(a - b)^2\Norm{P}^4}{N^2}\Norm{P}^2.
\end{split}
\end{equation}
Therefore, we have only to calculate $\ideal{[\mu(P), J(P)], P}$. 
One can easily see that 
\begin{equation}\label{eq:3.4}
[J(P), P] = [[Z, P], P] = [P, [P, Z]] = 2\mu(P). 
\end{equation}
This concludes that 
\begin{equation}\label{eq:3.33}
\ideal{[\mu(P), J(P)], P} 
= \ideal{\mu(P), [J(P), P]} 
= \ideal{\mu(P), 2 \mu(P)} = 2 \Norm{\mu(P)}^2 , 
\end{equation}
which completes the proof.
\end{proof}

%
%

\subsection{Preliminaries on root systems}
\label{subsection:root}

To calculate the Laplacian of $f_{a,b}(P)$, we need some properties of root systems. 
In this subsection, we recall the root systems of symmetric spaces of compact type $G/K$, not necessarily Hermitian. 
We refer to \cite{Helgason, Loos} for details. 

Let $\g = \k \oplus \p$ be the Cartan decomposition. 
Take a maximal abelian subspace $\a$ in $\p$. 
For $\alpha \in \a^\ast$, where $\a^\ast$ is the dual vector space of $\a$, let us define 
\begin{equation}
\g_{\alpha} := \{ 
X \in \g \mid (\ad_H)^2(X) = - \alpha(H)^2 X \ \ (\forall H \in \a) 
\} . 
\end{equation}
Note that $\g_{\alpha} = \g_{-\alpha}$ by definition. 
We call $\alpha \in \a^\ast$ a \textit{root} if it satisfies $\alpha \neq 0$ and $\g_{\alpha} \neq 0$.  
The set of roots, denoted by $\Delta$, is called the \textit{root system}. 
For each $\alpha \in \Delta \cup \{ 0 \}$, put 
\begin{equation}
\k_{\alpha} := \k \cap \g_{\alpha} , \quad 
\p_{\alpha} := \p \cap \g_{\alpha} . 
\end{equation}
These $\g_{\alpha}$, $\k_{\alpha}$ and $\p_{\alpha}$ are called the \textit{root spaces}. 
The following properties of root spaces are well-known (we refer to \cite[Chapter~VI, Proposition~1.4]{Loos}, 
\cite[Chapter~VII, Lemmas~11.3, 11.4]{Helgason}). 

\begin{Prop}
\label{Prop:Loos}
The root spaces satisfy the following. 
\begin{enumerate}
\item
We have the following decompositions$:$ 
\begin{equation}\label{eq:3.12}
\k = \k_0 \oplus \bigoplus_{\alpha} \k_{\alpha}, \quad 
\p = \a \oplus \bigoplus_{\alpha} \p_{\alpha} . 
\end{equation}
\item
For all $\alpha \in \Delta$, one has $\dim \k_{\alpha} = \dim \p_{\alpha}$. 
\item
For all $\alpha, \beta \in \Delta \cup \{ 0 \}$, one has 
\begin{equation}
\begin{split} 
[\k_{\alpha} , \k_{\beta}] & \subset \k_{\alpha + \beta} + \k_{\alpha - \beta} , \\ 
[\k_{\alpha} , \p_{\beta}] & \subset \p_{\alpha + \beta} + \p_{\alpha - \beta} , \\
[\p_{\alpha} , \p_{\beta}] & \subset \k_{\alpha + \beta} + \k_{\alpha - \beta} . 
\end{split} 
\end{equation}
\end{enumerate}
\end{Prop}

The decompositions given in \eqref{eq:3.12} are called the \textit{root space decompositions with respect to $\a$}. 
To be exact, the indices $\alpha$ run through $\Delta^+$, say a set of positive roots, since $\g_{\alpha} = \g_{- \alpha}$. 

\begin{Dfn}
The \textit{multiplicity} $m_\alpha$ of a root $\alpha$ is defined by 
\begin{equation}
m_{\alpha} := \dim \k_{\alpha} = \dim \p_{\alpha} . 
\end{equation}
\end{Dfn}

By using the above properties of root spaces, one can get the following lemma. 
We use this to compute the Laplacian of $f_{a, b}(P)$, and also in the latter section. 

\begin{Lem}\label{Lem:3.3}
For all $P \in \p$, we have
\begin{equation}\label{eq:3.7} 
\Norm{P}^2 = -2 \Tr\left( \left.(\ad_P)^2\right|_\p \right) . 
\end{equation}
\end{Lem}

\begin{proof}
We fix an arbitrary $P \in \p$. 
Since $(\ad_P)^2$ preserves $\k$ and $\p$ respectively, we have
\begin{equation}\label{eq:3.10}
\begin{split}
\Norm{P}^2 = -B(P, P) = -\Tr\left( (\ad_P)^2 \right) 
= -\Tr\left( \left.(\ad_P)^2\right|_\k \right) - \Tr\left( \left.(\ad_P)^2\right|_\p \right). 
\end{split}
\end{equation}
Thus, it is sufficient to show
\begin{equation}\label{eq:3.11}
\Tr\left( \left.(\ad_P)^2\right|_\k \right) = \Tr\left( \left.(\ad_P)^2\right|_\p \right). 
\end{equation}
In order to prove \eqref{eq:3.11}, we take a maximal abelian subspace $\a$ in $\p$ containing $P$, and the root space decomposition with respect to $\a$. 
Note that, since $P \in \a = \p_0$, Proposition~\ref{Prop:Loos} (3) yields that $(\ad_P)^2$ preserves each root space. 
Hence we obtain 
\begin{align}
\Tr\left( \left. \left( \ad_P \right)^2 \right|_{\k} \right) &= \Tr\left( \left. \left( \ad_P \right)^2 \right|_{\k_0} \right) + \sum_{\alpha} \Tr\left( \left. \left( \ad_P \right)^2 \right|_{\k_\alpha} \right), \label{eq:2.13}\\
\Tr\left( \left. \left( \ad_P \right)^2 \right|_{\p} \right) &= \Tr\left( \left. \left( \ad_P \right)^2 \right|_{\a} \right) + \sum_{\alpha} \Tr\left( \left. \left( \ad_P \right)^2 \right|_{\p_\alpha} \right). \label{eq:2.14}
\end{align}
By definition of the root spaces, one has 
\begin{equation}
\begin{split}
&\left. \left( \ad_P \right)^2 \right|_{\k_0} = 0,\qquad \left. \left( \ad_P \right)^2 \right|_{\a} = 0, \\ 
&\left. \left( \ad_P \right)^2 \right|_{\k_\alpha} = -\alpha(P)^2 \id_{\k_\alpha},\qquad \left. \left( \ad_P \right)^2 \right|_{\p_\alpha} = -\alpha(P)^2 \id_{\p_\alpha}. 
\end{split}
\end{equation}
Therefore, it follows from Proposition~\ref{Prop:Loos} (2) that
\begin{equation}\label{eq:3.15}
\begin{split}
\Tr\left( \left. \left( \ad_P \right)^2 \right|_{\k_0} \right) 
&= \Tr\left( \left. \left( \ad_P \right)^2 \right|_{\a} \right) = 0,\\ 
\Tr\left( \left. \left( \ad_P \right)^2 \right|_{\k_\alpha} \right) 
&= -\alpha(P)^2 \dim \k_\alpha 
= -\alpha(P)^2 \dim \p_\alpha 
= \Tr\left( \left. \left( \ad_P \right)^2 \right|_{\p_\alpha} \right). 
\end{split}
\end{equation}
By substituting \eqref{eq:3.15} for \eqref{eq:2.13} and \eqref{eq:2.14}, we obtain \eqref{eq:3.11}. 
Thus we complete the proof. 
\end{proof}

%
%

\subsection{The Laplacian}

In this subsection, we compute the Laplacian of $f_{a, b}(P)$. 
Recall that the Laplacian is given by
\begin{equation}
\Delta f_{a, b}(P) = \sum \dfrac{\partial^2 f_{a, b}(P)}{\partial {P_i}^2} . 
\end{equation}
Recall that $f_{a, b}(P)$ can be written as a linear combination 
of $\Norm{P}^4$ and $\Norm{\mu(P)}^2$. 
Hence we start from calculating their partial differentials of order two. 

\begin{Lem}\label{Lem:3.8}
Partial differentials of $\Norm{P}^4$ and $\Norm{\mu(P)}^2$ of order two satisfy 
\begin{enumerate}
\item $(\partial^2 \Norm{P}^4) / (\partial {P_i}^2) = 8\ideal{P, P_i}^2 + 4\Norm{P}^2$, 
\item $(\partial^2 \Norm{\mu(P)}^2) / (\partial {P_i}^2) = 2\ideal{\ad_{\mu(P)}\circ\ad_Z(P_i), P_i} - 2\ideal{ \left( \ad_{J(P)} \right)^2 (P_i), P_i}$. 
\end{enumerate}
\end{Lem}

\begin{proof}
First, we prove (1). 
One knows the first differentials from Lemma \ref{Lem:3.4} (1). 
Hence we have 
\begin{equation}\label{eq:3.35}
\begin{split}
\dfrac{\partial^2 \Norm{P}^4}{\partial {P_i}^2} 
= \dfrac{\partial }{\partial P_i}4\Norm{P}^2\ideal{P, P_i} 
= 4\left( \left( \dfrac{\partial \Norm{P}^2}{\partial P_i} \right) \ideal{P, P_i} + \Norm{P}^2\dfrac{\partial \ideal{P, P_i}}{\partial P_i} \right). 
\end{split}
\end{equation}
From \eqref{eq:3.20} and \eqref{eq:3.19}, one has 
\begin{equation}\label{eq:3.36}
\dfrac{\partial \Norm{P}^2}{\partial P_i} = 2 \ideal{P, P_i} , \quad 
\dfrac{\partial \ideal{P, P_i}}{\partial P_i} 
= \ideal{\dfrac{\partial P}{\partial P_i} , P_i}
= \ideal{P_i, P_i} = 1 . 
\end{equation}
Our first claim (1) has been shown by substituting \eqref{eq:3.36} for \eqref{eq:3.35}. 

Next, we show (2). 
We use the first differentials calculated in Lemma \ref{Lem:3.4} (2). 
By direct computations, one has 
\begin{equation}\label{eq:3.38}
\begin{split}
\dfrac{\partial^2 \Norm{\mu(P)}^2}{\partial {P_i}^2} 
&= \dfrac{\partial }{\partial P_i}\left( 2\ideal{[\mu(P), J(P)], P_i} \right) \\ 
&= 
2\ideal{[\mu(P), \dfrac{\partial J(P)}{\partial P_i}], P_i} 
+ 2\ideal{[\dfrac{\partial \mu(P)}{\partial P_i}, J(P)], P_i} . 
\end{split}
\end{equation}
It follows from \eqref{eq:3.23} that 
\begin{equation}\label{eq:3.39}
\begin{split}
\dfrac{\partial J(P)}{\partial P_i} 
= \dfrac{\partial [Z, P]}{\partial P_i} 
= [Z, \dfrac{\partial P}{\partial P_i}] 
= [Z, P_i] , \quad 
\dfrac{\partial \mu(P)}{\partial P_i} = [J(P), P_i] . 
\end{split}
\end{equation}
By substituting \eqref{eq:3.39} for \eqref{eq:3.38}, one concludes the proof of (2). 
\end{proof}

From the above calculations of differentials of order two, we can compute the Laplacians of $\Norm{P}^4$ and $\Norm{\mu(P)}^2$. 

\begin{Lem}
\label{lem:laplacian}
We have 
\begin{enumerate}
\item 
$\Delta \Norm{P}^4 = 4(N + 2)\Norm{P}^2$, 
\item 
$\Delta \Norm{\mu(P)}^2 = 2\Norm{P}^2$. 
\end{enumerate}
\end{Lem}

\begin{proof}
We prove (1). 
By Lemma \ref{Lem:3.8} (1), we find
\begin{equation}\label{eq:3.51}
\Delta \Norm{P}^4 
= \sum \dfrac{\partial^2 \Norm{P}^4}{\partial {P_i}^2} 
= 8\sum \ideal{P, P_i}^2 + 4 \sum \Norm{P}^2. 
\end{equation}
The first term of \eqref{eq:3.51} satisfies 
\begin{equation}\label{eq:3.52}
8\sum \ideal{P, P_i}^2 = 8\sum \ideal{P, \ideal{P, P_i}P_i} 
= 8\ideal{P, \sum \ideal{P, P_i}P_i} = 8 \Norm{P}^2. 
\end{equation}
The second term of \eqref{eq:3.51} satisfies 
\begin{equation}\label{eq:3.53}
\textstyle 
4 \sum \Norm{P}^2 = 4 \sum_{i = 1}^{N} \Norm{P}^2 = 4 N \Norm{P}^2 . 
\end{equation}
We thus complete the proof of (1). 

We prove (2). 
Lemma \ref{Lem:3.8} (2) yields that 
\begin{equation}\label{eq:3.45}
\begin{split}
\Delta \Norm{\mu(P)}^2 
= \sum_{i = 1}^{N} \dfrac{\partial^2 \Norm{\mu(P)}^2}{\partial {P_i}^2} 
= 2 \Tr \left( \left. \ad_{\mu(P)} \circ \ad_Z \right|_\p \right) 
-2 \Tr \left( \left. \left( \ad_{J(P)} \right)^2 \right|_\p \right) . 
\end{split}
\end{equation}
In order to calculate each term of \eqref{eq:3.45}, we need
\begin{equation}
\begin{split}
\Norm{J(P)}^2 
& = \ideal{[Z, P], [Z, P]} 
= -\ideal{[Z, [Z, P]], P} 
= -\ideal{-P, P} = \Norm{P}^2 , \\ 
\Norm{J(P)}^2 
& = \ideal{[Z, P], [Z, P]} 
= - \ideal{[P, [P, Z]], Z} 
= -2 \ideal{\mu(P), Z} . 
\end{split}
\end{equation}
Thus, since $\ad_Z |_\k = 0$, 
the first term of \eqref{eq:3.45} satisfies 
\begin{equation}
2 \Tr \left( \left. \ad_{\mu(P)} \circ \ad_Z \right|_\p \right) 
= 2 \Tr \left( \ad_{\mu(P)} \circ \ad_Z \right) 
= 2 B(\mu(P) , Z) 
= -2 \ideal{\mu(p) , Z} 
= \Norm{P}^2 . 
\end{equation}
From Lemma \ref{Lem:3.3}, the second term of \eqref{eq:3.45} satisfies 
\begin{equation}\label{eq:3.47}
-2 \Tr \left( \left. \left( \ad_{J(P)} \right)^2 \right|_\p \right) 
= \Norm{J(P)}^2 = \Norm{P}^2 . 
\end{equation}
This concludes the proof of (2). 
\end{proof}

The following proposition gives us a formula for the Laplacian of $f_{a, b}(P)$. 

\begin{Prop}\label{Prop:3.9}
We have
\begin{equation}\label{eq:3.41}
\Delta f_{a, b}(P) = \dfrac{(N + 2) a + (N - 2) b}{N} \Norm{P}^2.
\end{equation}  
\end{Prop}

\begin{proof}
We use the formula of $f_{a,b}(P)$ given in Proposition \ref{Prop:f-ab}. 
By linearity of Laplacian, we find 
\begin{equation}\label{eq:3.42}
\Delta f_{a, b}(P) = b \Delta \Norm{\mu(P)}^2 + \dfrac{a - b}{4 N} \Delta \Norm{P}^4. 
\end{equation}
Thus, the proof easily follows from Lemma \ref{lem:laplacian}. 
\end{proof}

%
%
\section{The gradient and the Laplacian in the case of rank two}\label{sect:rank two}

In this section, we compute the squared-norm of the gradient and the Laplacian of $f_{a, b}(P)$ in the case of $\rank G / K = 2$. 

\subsection{Preliminaries on root systems for Hermitian case}
\label{subsection:root-hermitian}

In this subsection, we relate the element $Z$ to the root space decomposition. 
In order to do this, we review some properties of the root systems of $G/K$, irreducible Hermitian symmetric spaces of rank $r$. 

As in Subsection~\ref{subsection:root}, we denote a maximal abelian subspace of $\p$ by $\a$, and consider the root system $\Delta$ with respect to $\a$. 
Recall that the root system of type $C_r$ or $BC_r$ can be represented as 
\begin{equation}
\begin{split}
\Delta(C_r) & = 
\left\{ \pm(\varepsilon_i - \varepsilon_j), \pm(\varepsilon_i + \varepsilon_j) \bigm| 1\leq i < j \leq r \right\} \cup \left\{ \pm 2\varepsilon_i \bigm| 1\leq i \leq r\right\}, \\ 
\Delta(BC_r) & = 
\left\{ \pm(\varepsilon_i - \varepsilon_j), \pm(\varepsilon_i + \varepsilon_j) \bigm| 1\leq i < j \leq r \right\} \cup \left\{ \pm\varepsilon_i, \pm 2\varepsilon_i \bigm| 1\leq i \leq r\right\}. 
\end{split}
\end{equation}
Note that $\left\{ \varepsilon_i \bigm| i = 1, 2,\dotsc, r \right\}$ is a basis of the dual space $\a^\ast$ of $\a$. 
The next proposition follows from the classification of compact irreducible Hermitian symmetric spaces (see \cite{Helgason}, or the table in \cite{Tamaru2}). 
A classification-free proof can be found in \cite[Theorem 4.1]{Nagano--Tanaka}. 

\begin{Prop}
Let $G/K$ be a compact irreducible Hermitian symmetric space. 
Then we have 
\begin{enumerate}
\item
$\Delta$ is of type $C_r$ or $BC_r$, 
\item
the multiplicities satisfy 
$m_{2 \varepsilon_i} = 1$ for all $i \in \left\{ 1, 2,\dotsc, r \right\}$. 
\end{enumerate}
\end{Prop}

In the latter arguments, the root spaces $\k_{2 \varepsilon_i}$ play important roles. 
We here recall some of their properties. 

\begin{Lem}
\label{Lem:k-epsilon}
Under the above settings, we have 
\begin{enumerate}
\item
$[\k_0 , \k_{2 \varepsilon_i}] = 0$ for all $i$, 
\item
$[\k_{2 \varepsilon_i} , \k_{2 \varepsilon_j}] = 0$ for all $i, j$, 
\item
$[\k_{2 \varepsilon_i} , \p_{2 \varepsilon_j}] = 0$ for all $i \neq j$. 
\end{enumerate}
\end{Lem}

\begin{proof}
The claim (1) follows from the skew-symmetry of the adjoint action of $\k_0$ on 
$\k_{2 \varepsilon_i}$ and $\dim \k_{2 \varepsilon_i} = 1$. 
We show (2). 
If $i=j$, then again $\dim \k_{2 \varepsilon_i} = 1$ completes the proof. 
If $i \neq j$, then the assertion follows from 
\begin{equation}
[\k_{2 \varepsilon_i} , \k_{2 \varepsilon_j}] 
\subset 
\k_{2 \varepsilon_i + 2 \varepsilon_j} +  \k_{2 \varepsilon_i - 2 \varepsilon_j} 
= 0 . 
\end{equation}
Note that the first inclusion follows from 
the property of root spaces (Proposition~\ref{Prop:Loos} (3)), 
and second equality follows from 
the fact that $2 \varepsilon_i \pm 2 \varepsilon_j$ are not roots. 
The claim (3) can be proved similarly. 
\end{proof}

From the above properties, one can determine a position of the element $Z$. 
The following proposition was given by Song (\cite[Lemma 2.2 (ii)]{Song}), and also can be found in the proof of \cite[Proposition 4.4 ]{Tamaru3}. 

\begin{Prop}[\cite{Song, Tamaru3}]
\label{Prop:4.1}
We have $Z \in \k_0 \oplus \k_{2\varepsilon_1} \oplus\dots\oplus \k_{2\varepsilon_r}$.  
\end{Prop}

\begin{proof}
We here describe a sketch of the proof. 
Let $\k'_0$ be a maximal abelian subalgebra in $\k_0$. 
Then $\k'_0 \oplus \a$ is abelian. 
Furthermore, a standard argument shows that 
$\k'_0 \oplus \a$ is maximal abelian in $\g$. 
Hence it is a Cartan subalgebra. 
This yields that 
\begin{equation}
\rank \g = \dim (\k'_0 \oplus \a) = \dim \k'_0 + r. 
\end{equation}
On the other hand, Lemma \ref{Lem:k-epsilon} yields that
\begin{equation}
\h' := \k'_0 \oplus \k_{2\varepsilon_1} \oplus \dotsb \oplus \k_{2\varepsilon_r} 
\end{equation}
is abelian. 
Recall that $\dim \k_{2\varepsilon_i} = 1$. 
Thus, from dimension reasons, $\h'$ is also a Cartan subalgebra of $\g$, and hence maximal abelian in $\k$. 
Since $Z \in C(\k)$, we conclude $Z \in \h'$, which completes the proof. 
\end{proof}

\subsection{A decomposition of $Z$}

From Proposition~\ref{Prop:4.1}, we have a decomposition
\begin{equation}\label{eq:4.4}
Z = Z_0 + Z_{2\varepsilon_1} +\dots+ Z_{2\varepsilon_r}, 
\end{equation}
where $Z_0 \in \k_0$ and $Z_{2\varepsilon_i} \in \k_{2\varepsilon_i}$. 
In this subsection, we compute the length of each factor $Z_{2\varepsilon_i}$. 
Assume that $\rank G / K = r$ and keep the notations in the previous subsection. 

First of all, we recall root vectors. 
A vector $H_{\alpha} \in \a$ is called the \textit{root vector corresponding to a root $\alpha$} if it holds that
\begin{equation}
\ideal{H_{\alpha}, H} = \alpha(H) \quad (\forall H \in \a). 
\end{equation}
We here mention well-known properties of root vectors of the root system of type $C_r$ or $BC_r$. 
We refer to an explicit description of each root system, for example, in \cite{Helgason, Humphreys, Loos}. 

\begin{Lem}
\label{Lem:root-vector}
Let $G/K$ be a compact irreducible symmetric space with root system of type $C_r$ or $BC_r$. 
Then we have 
\begin{enumerate}
\item
$\Norm{H_{\varepsilon_1}}^2 = \cdots = \Norm{H_{\varepsilon_r}}^2$, 
\item
$\ideal{H_{\varepsilon_i} , H_{\varepsilon_j}} = 0$ if $i \neq j$. 
\end{enumerate}
\end{Lem}

We now study each factor $Z_{2\varepsilon_i}$ of $Z$. 
In particular, we express the lengths of $Z_{2 \varepsilon_i}$ in terms of the lengths of root vectors. 

\begin{Prop}\label{Prop:4.3}
For all $i \in \left\{ 1, 2,\dotsc, r \right\}$, 
we have  
\begin{enumerate}
\item
$( \ad_{Z_{2\varepsilon_i}})^2\left( H_{\varepsilon_i} \right) = -H_{\varepsilon_i}$, 
\item
$4\Norm{H_{\varepsilon_i}}^2\Norm{Z_{2\varepsilon_i}}^2 = 1$. 
\end{enumerate}
\end{Prop}

\begin{proof}
We show (1). 
By definition of $Z$, one has 
\begin{equation} 
- H_{\varepsilon_i} = (\ad_{Z})^2\left( H_{\varepsilon_i} \right). 
\end{equation}
Lemma~\ref{Lem:root-vector} yields that $2 \varepsilon_i (H_{\varepsilon_j}) = 0$ if $i \neq j$. 
We thus have
\begin{equation}
\ad_{Z} \left( H_{\varepsilon_i} \right) 
= \left[ Z_0 + Z_{2 \varepsilon_1} + \cdots + Z_{2 \varepsilon_r} , H_{\varepsilon_i} \right] 
= \left[ Z_{2 \varepsilon_i} , H_{\varepsilon_i} \right] 
\in \p_{2 \varepsilon_i}. 
\end{equation}
By Lemma \ref{Lem:k-epsilon}, one knows $[\k_{2 \varepsilon_i} , \p_{2 \varepsilon_j}] = 0$ for $i \neq j$. 
Therefore, we have 
\begin{equation}
(\ad_{Z})^2 \left( H_{\varepsilon_i} \right) 
= \left[ Z_0 + Z_{2 \varepsilon_1} + \cdots + Z_{2 \varepsilon_r} , 
\left[ Z_{2 \varepsilon_i} , H_{\varepsilon_i} \right] \right] 
= (\ad_{Z_{2\varepsilon_i}})^2 \left( H_{\varepsilon_i} \right) , 
\end{equation}
which concludes the proof of (1). 

We show (2). 
It follows from (1) that
\begin{equation}\label{eq:4.6}
\Norm{H_{\varepsilon_i}}^2 
= -\ideal{\left( \ad_{Z_{2\varepsilon_i}} \right)^2\left( H_{\varepsilon_i} \right), H_{\varepsilon_i}} 
= -\ideal{Z_{2\varepsilon_i}, \left[ H_{\varepsilon_i}, \left[ H_{\varepsilon_i}, Z_{2\varepsilon_i}\right] \right]} . 
\end{equation}
By definition of roots, one has 
\begin{equation}\label{eq:4.8}
\left[ H_{\varepsilon_i}, \left[ H_{\varepsilon_i}, Z_{2\varepsilon_i}\right] \right] 
= -\left( 2\varepsilon_i \left( H_{\varepsilon_i} \right) \right)^2 Z_{2\varepsilon_i} 
= -4 \Norm{ H_{\varepsilon_i} }^4 Z_{2\varepsilon_i} . 
\end{equation}
By substituting \eqref{eq:4.8} for \eqref{eq:4.6}, we get 
\begin{equation}
\Norm{H_{\varepsilon_i}}^2 
= 4 \Norm{ H_{\varepsilon_i} }^4 \Norm{ Z_{2\varepsilon_i} }^2 . 
\end{equation}
Since $H_{\varepsilon_i} \neq 0$, our assertion is proved. 
\end{proof}

%
%
\subsection{Preliminaries on root systems for rank two Hermitian case}
\label{Sec:rank 2}

Hereafter, let $G / K$ be a compact irreducible Hermitian symmetric space of rank two. 
In this subsection, we give formulas for $N := \dim \p$ and $\Norm{H_{\varepsilon_i}}^2$ in terms of the roots and the multiplicities. 

Recall that the root system $\Delta$ of $G / K$ is of type $BC_2$ or $C_2$. 
For the case of type $BC_2$, as mentioned in Subsection~\ref{subsection:root-hermitian}, $\Delta$ is given by
\begin{equation} 
\Delta = \Delta(BC_2) 
= \pm\left\{ \varepsilon_1,\ \varepsilon_2,\ 2\varepsilon_1,\ 2\varepsilon_2,\ \varepsilon_1 + \varepsilon_2,\ \varepsilon_1 - \varepsilon_2 \right\} . 
\end{equation} 
Since an element $\sigma$ of the Weyl group induces a linear isomorphism from $\k_{\alpha}$ onto $\k_{\sigma(\alpha)}$, the multiplicities $m_{\alpha}$ of roots are invariant under the action of the Weyl group. 
Hence we have
\begin{equation} 
m_{\varepsilon_1 + \varepsilon_2} = m_{\varepsilon_1 - \varepsilon_2}, \quad 
m_{\varepsilon_1} = m_{\varepsilon_2} , \quad 
m_{2\varepsilon_1} = m_{2\varepsilon_2} = 1 .
\end{equation}
We regard $\Delta(C_2)$ as $\Delta(BC_2)$ with $m_{\varepsilon_1} = m_{\varepsilon_2} = 0$. 
We thus consider $\Delta(BC_2)$ only. 
To make the latter convenient, we put 
\begin{equation} 
\label{eq:multiplicities}
m_1 := m_{\varepsilon_1 + \varepsilon_2} , \quad 
m_2 := m_{\varepsilon_1} + m_{2 \varepsilon_1} = m_{\varepsilon_1} + 1 . 
\end{equation}
Note that $m_1$ and $m_2$ will be the multiplicities of the principal curvatures of the isoparametric hypersurfaces in spheres obtained from the isotropy representations of $G/K$. 

We first give a formula for $N = \dim \p$, which follows easily. 

\begin{Lem}
\label{Lem:length-N}
We have $N = 2 m_{1} + 2 m_{2} + 2$. 
\end{Lem}

\begin{proof}
This is a direct consequence of the root space decomposition of $\p$. 
\end{proof}

We next compute the length of the root vectors 
$H_{\varepsilon_1}$ and $H_{\varepsilon_2}$. 

\begin{Lem}\label{Lem:4.4} 
We have 
$1 / {\Norm{H_{\varepsilon_1}}^2} = 1 / {\Norm{H_{\varepsilon_2}}^2} 
= 2(2m_{1} + m_{2} + 3)$. 
\end{Lem}

\begin{proof}
One knows $\Norm{H_{\varepsilon_1}}^2 = \Norm{H_{\varepsilon_2}}^2$ from Lemma \ref{Lem:root-vector}. 
Thus it is sufficient to compute $\Norm{H_{\varepsilon_1}}^2$. 

First of all, we derive a general formula of $\Norm{H}^2$ for $H \in \a$. 
By Lemma \ref{Lem:3.3}, we have
\begin{equation}
\Norm{H}^2 = -2 \Tr\left(\left. ( \ad_{H} )^2 \right|_\p \right) 
= -2 \sum_{j} \ideal{\left( \ad_{H} \right)^2(P_j), P_j} , 
\end{equation}
where $\{ P_j \}$ is an orthonormal basis of $\p$. 
Since the representation of the right-hand side is independent of 
the choice of an orthonormal basis of $\p$, 
we can assume each $P_j$ is a unit vector in $\a$ or 
some root space $\p_{\alpha}$. 
For a unit vector $P_\alpha \in \p_\alpha$, we have
\begin{equation}
\ideal{\left( \ad_{ H } \right)^2 (P_\alpha), P_\alpha} 
= \ideal{- \alpha( H )^2 P_\alpha , P_\alpha} = - \alpha( H )^2 . 
\end{equation}
Hence, it follows that 
\begin{equation}
\Norm{H}^2 
= -2 \sum_{i} \ideal{\left( \ad_{H} \right)^2(P_i), P_i} 
= 2\sum_{\alpha} m_{\alpha} \alpha(H)^2. 
\end{equation}
Note that this formula is true for every symmetric space. 

We use a notation $\ideal{\alpha , \beta} := \alpha(H_{\beta})$. 
Then, since the root system is of $BC_2$-type, it follows 
\begin{equation}
\begin{split}
\Norm{H_{\varepsilon_1}}^2 &= 2\left( m_{\varepsilon_1 - \varepsilon_2}\ideal{\varepsilon_1 - \varepsilon_2, \varepsilon_1}^2 + m_{\varepsilon_2}\ideal{\varepsilon_2, \varepsilon_1}^2 + m_{2\varepsilon_2}\ideal{2\varepsilon_2, \varepsilon_1}^2 \right.\\
&\hspace{30pt}\left. m_{\varepsilon_1}\ideal{\varepsilon_1, \varepsilon_1}^2 + m_{\varepsilon_1 + \varepsilon_2}\ideal{\varepsilon_1 + \varepsilon_2, \varepsilon_1}^2 + m_{2\varepsilon_1}\ideal{2\varepsilon_1, \varepsilon_1}^2 \right)\\
&= 2\left( 
m_1 + (m_2 - 1) + m_1 + 4 
\right) \ideal{\varepsilon_1, \varepsilon_1}^2\\
&= 2\left( 
2 m_1 + m_2 + 3 
\right) \Norm{H_{\varepsilon_1}}^4 . 
\end{split}
\end{equation}
Since $H_{\varepsilon_1} \neq 0$, this completes the proof. 
\end{proof}

%
%
\subsection{The squared-norm of the gradient}

According to Proposition \ref{Prop:3.7}, the squared-norm $\Norm{\grad f_{a, b}(P)}^2$ of the gradient can be written as a linear combination of $\Norm{[J(P), \mu(P)]}^2$, $\Norm{P}^2\Norm{\mu(P)}^2$ and $\Norm{P}^6$. 
In this subsection, we simplify this presentation in the case of $\rank G / K = 2$. 

We need formulas for $\Norm{[J(P), \mu(P)]}^2$, $\Norm{P}^2\Norm{\mu(P)}^2$ and $\Norm{P}^6$ when $P \in \a$. 
First of all, we give a formula for $\Norm{P}^6$. 

\begin{Lem}\label{Lem:P6}
For $P = a_1 H_{\varepsilon_1} + a_2 H_{\varepsilon_2} \in \a$, 
we have 
\begin{enumerate}
\item
$\Norm{P}^2 = ({a_1}^2 + {a_2}^2) \Norm{H_{\varepsilon_1}}^2$, 
\item
$\Norm{P}^6 = ( {a_1}^6 + 3 {a_1}^4 {a_2}^2 + 3 {a_1}^2 {a_2}^4 + {a_2}^6 ) 
\Norm{H_{\varepsilon_1}}^6$. 
\end{enumerate}
\end{Lem}

\begin{proof}
The first claim follows from $\ideal{H_{\varepsilon_1} , H_{\varepsilon_2}} = 0$ and $\Norm{H_{\varepsilon_1}}^2 = \Norm{H_{\varepsilon_2}}^2$ (see Lemma \ref{Lem:root-vector}). 
The second claim follows easily from the first one. 
\end{proof}

Next we give a formula for $\Norm{P}^2\Norm{\mu(P)}^2$. 

\begin{Lem}
\label{Lem:P2mu2}
For $P = a_1 H_{\varepsilon_1} + a_2 H_{\varepsilon_2} \in \a$, 
we have 
\begin{enumerate}
\item
$\mu(P) = -2 ({a_1}^2 Z_{2\varepsilon_1} + {a_2}^2 Z_{2\varepsilon_2}) 
\Norm{H_{\varepsilon_1}}^4$, 
\item
$\Norm{P}^2 \Norm{\mu(P)}^2 
= ({a_1}^6 + {a_1}^4 {a_2}^2 + {a_1}^2 {a_2}^4 + {a_2}^6) 
\Norm{H_{\varepsilon_1}}^8$. 
\end{enumerate}
\end{Lem}

\begin{proof}
We show (1). 
By the decomposition $Z = Z_0 + Z_{2\varepsilon_1} + Z_{2\varepsilon_2}$, it follows that 
\begin{equation}\label{eq:4.23}
2 \mu(P) 
= [P, [P, Z]] 
= [P, [P, Z_0]] + [P, [P, Z_{2\varepsilon_1}]] + [P, [P, Z_{2\varepsilon_2}]] . 
\end{equation}
Since $P \in \a$, the definition of roots yields that 
\begin{equation}\label{eq:4.24}
\begin{split}
[P, [P, Z_0]] & = 0 , \\ 
[P, [P, Z_{2\varepsilon_1}]] 
& = - \left( 2\varepsilon_1(P) \right)^2 Z_{2\varepsilon_1} 
= -4 {a_1}^2 \Norm{H_{\varepsilon_1}}^4 Z_{2\varepsilon_1} , \\ 
[P, [P, Z_{2\varepsilon_2}]] 
& = - \left( 2\varepsilon_2(P) \right)^2 Z_{2\varepsilon_2} 
= -4 {a_2}^2 \Norm{H_{\varepsilon_1}}^4 Z_{2\varepsilon_2} . 
\end{split}
\end{equation}
By substituting \eqref{eq:4.24} for \eqref{eq:4.23}, we conclude the proof of (1). 

We show (2). 
Recall that different root spaces are orthogonal to each other (cf. \cite[Chapter~III, Theorem~4.2]{Helgason}). 
This yields that 
\begin{equation}
\ideal{Z_{2\varepsilon_1} , Z_{2\varepsilon_2}} = 0 . 
\end{equation}
Then, it follows from (1) and Proposition \ref{Prop:4.3} that 
\begin{equation}
\Norm{ \mu(P) }^2 
= 4 \Norm{H_{\varepsilon_1}}^8 \left(
{a_1}^4 \Norm{Z_{2\varepsilon_1}}^2 + {a_2}^4 \Norm{Z_{2\varepsilon_2}}^2
\right) = ({a_1}^4 + {a_2}^4) \Norm{H_{\varepsilon_1}}^6 . 
\end{equation}
This and Lemma \ref{Lem:P6} (1) conclude the proof of (2). 
\end{proof}

We then give a formula for $\Norm{ [J(P), \mu(P)] }^2$. 

\begin{Lem}\label{Lem:Jmu}
For $P = a_1 H_{\varepsilon_1} + a_2 H_{\varepsilon_2} \in \a$, we have 
\begin{equation}
\Norm{ [J(P), \mu(P)] }^2 
= 4 \left( {a_1}^6 + {a_2}^6 \right) \Norm{H_{\varepsilon_1}}^{10} . 
\end{equation}
\end{Lem}

\begin{proof}
Recall that $[\k_0 , \a] = 0$ and $\ideal{H_{\varepsilon_1} , H_{\varepsilon_2}} = 0$ (see Lemma~\ref{Lem:root-vector}). 
Then we have 
\begin{equation}\label{eq:4.29}
J(P) 
= [ Z_0 + Z_{2\varepsilon_1} + Z_{2\varepsilon_2} , 
a_1 H_{\varepsilon_1} + a_2 H_{\varepsilon_2} ] 
= a_1 [Z_{2\varepsilon_1} , H_{\varepsilon_1} ] 
+ a_2 [Z_{2\varepsilon_2} , H_{\varepsilon_2} ] . 
\end{equation}
Lemma~\ref{Lem:P2mu2} (1) yields that
\begin{equation}
[J(P), \mu(P)] 
= \left[ a_1 [Z_{2\varepsilon_1} , H_{\varepsilon_1} ] 
+ a_2 [Z_{2\varepsilon_2} , H_{\varepsilon_2} ] , \ 
- 2 ({a_1}^2 Z_{2\varepsilon_1} + 2{a_2}^2 Z_{2\varepsilon_2}) 
\Norm{H_{\varepsilon_1}}^4 \right] . 
\end{equation}
From Lemma~\ref{Lem:k-epsilon}, one has 
\begin{equation}
[[Z_{2\varepsilon_1} , H_{\varepsilon_1} ] , Z_{2\varepsilon_2} ] 
= [[Z_{2\varepsilon_2} , H_{\varepsilon_2} ] , Z_{2\varepsilon_1} ] 
= 0 . 
\end{equation}
Therefore, Proposition~\ref{Prop:4.3} (1) yields that 
\begin{equation}\label{eq:4.30}
\begin{split}
[J(P), \mu(P)] 
& = 
- 2 \Norm{H_{\varepsilon_1}}^4  \left(
{a_1}^3 [[Z_{2\varepsilon_1} , H_{\varepsilon_1}], Z_{2\varepsilon_1}] 
+ {a_2}^3 [[Z_{2\varepsilon_2} , H_{\varepsilon_2}], Z_{2\varepsilon_2}] 
\right) \\ 
& = 
-2 \Norm{H_{\varepsilon_1}}^4  \left(
{a_1}^3 H_{\varepsilon_1} + {a_2}^3 H_{\varepsilon_2} 
\right) . 
\end{split}
\end{equation}
By taking the squared-norms, Lemma~\ref{Lem:P6} (1) concludes the proof. 
\end{proof}

By using these lemmas, $\Norm{[J(P), \mu(P)]}^2$ can be written as a linear combination of $\Norm{P}^2\Norm{\mu(P)}^2$ and $\Norm{P}^6$. 

\begin{Lem}\label{Lem:4.7}
For all $P \in \p$, we have 
\begin{equation}\label{eq:4.27}
\Norm{ \left[ J(P), \mu(P) \right] }^2 
= - 2 \Norm{H_{\varepsilon_1}}^4 \Norm{P}^6 
+ 6 \Norm{H_{\varepsilon_1}}^2 \Norm{P}^2\Norm{\mu(P)}^2 . 
\end{equation}
\end{Lem}

\begin{proof}
In general, if $f, g : \p \to \bR$ are $K$-invariant functions and satisfy $f |_{\a} = g |_{\a}$, then $f=g$ holds. 
Since the both sides of \eqref{eq:4.27} are $K$-invariant, we have only to prove that they coincide on $\a$. 
To show this, take any $P \in \a$, and write $P = a_1 H_{\varepsilon_1} + a_2 H_{\varepsilon_2}$. 
From Lemmas \ref{Lem:P6} and \ref{Lem:P2mu2}, one has 
\begin{equation}
\begin{split}
- 2 \Norm{H_{\varepsilon_1}}^4 \Norm{P}^6 
& = 
-2 ( {a_1}^6 + 3 {a_1}^4 {a_2}^2 + 3 {a_1}^2 {a_2}^4 + {a_2}^6 ) 
\Norm{H_{\varepsilon_1}}^{10} , \\ 
6 \Norm{H_{\varepsilon_1}}^2 \Norm{P}^2\Norm{\mu(P)}^2 
& = 
6 ({a_1}^6 + {a_1}^4 {a_2}^2 + {a_1}^2 {a_2}^4 + {a_2}^6) 
\Norm{H_{\varepsilon_1}}^{10} . 
\end{split}
\end{equation}
Therefore, the right-side hand of \eqref{eq:4.27} satisfies 
\begin{equation}
- 2 \Norm{H_{\varepsilon_1}}^4 \Norm{P}^6 
+ 6 \Norm{H_{\varepsilon_1}}^2 \Norm{P}^2\Norm{\mu(P)}^2 
= 4 \left( {a_1}^6 + {a_2}^6 \right) \Norm{H_{\varepsilon_1}}^{10} . 
\end{equation}
By Lemma \ref{Lem:Jmu}, 
this coincides with the left-hand side of \eqref{eq:4.27}. 
This completes the proof. 
\end{proof}

The following gives us a representation of the squared-norm of the gradient of $f_{a, b}$ in the case of rank two. 

\begin{Prop}\label{Prop:4.8}
For every $P \in \p$, we have 
\begin{equation}\label{eq:4.35}
\begin{split}
\Norm{ \grad f_{a, b}(P) }^2 &= \left( -\dfrac{2 b^2}{(2 m_{1} + m_{2} + 3)^2} + \dfrac{(a - b)^2}{4(m_{1} + m_{2} + 1)^2} \right) \Norm{P}^6\\ 
&\hspace{1cm}+ \left( \dfrac{12 b^2}{2 m_{1} + m_{2} + 3} + \dfrac{4 b (a - b)}{m_{1} + m_{2} + 1} \right) \Norm{P}^2\Norm{\mu(P)}^2. 
\end{split}
\end{equation}
\end{Prop}

\begin{proof}
From Proposition \ref{Prop:3.7}, one knows 
\begin{equation}
\Norm{\grad f_{a, b}(P)}^2 
= 4 b^2 \Norm{[J(P), \mu(P)]}^2 
+ \dfrac{8 b (a - b)}{N}\Norm{P}^2\Norm{\mu(P)}^2 
+ \dfrac{(a - b)^2}{N^2}\Norm{P}^6 . 
\end{equation}
By using Lemma \ref{Lem:4.7}, we have 
\begin{equation}
\Norm{\grad f_{a, b}(P)}^2 = A \Norm{P}^6 + B \Norm{P}^2\Norm{\mu(P)}^2 , 
\end{equation}
where 
\begin{equation}
\label{AB}
A = -2 \Norm{H_{\varepsilon_1}}^4 \cdot 4 b^2 + \dfrac{(a - b)^2}{N^2} , \quad 
B = 6 \Norm{H_{\varepsilon_1}}^2 \cdot 4 b^2 + \dfrac{8 b (a - b)}{N} . 
\end{equation}
From Lemmas \ref{Lem:length-N} and \ref{Lem:4.4}, one knows 
\begin{equation}
\Norm{H_{\varepsilon_1}}^2 = 1 / (2 (2 m_1 + m_2 + 3)) , \quad 
N = 2 m_1 + 2 m_2 + 2 . 
\end{equation}
By substituting these for \eqref{AB}, we conclude the proof. 
\end{proof}

\subsection{The Laplacian}

Next, we give a representation of $\Delta f_{a, b}(P)$. 

\begin{Prop}\label{Prop:4.9}
For all $P \in \p$, we have 
\begin{equation}\label{eq:4.38}
\Delta f_{a, b}(P) = \dfrac{(m_{1} + m_{2} + 2)a + (m_{1} + m_{2})b}{m_{1} + m_{2} + 1}\Norm{P}^2. 
\end{equation}
\end{Prop}

\begin{proof}
From Proposition~\ref{Prop:3.9}, one knows
\begin{equation}
\Delta f_{a, b}(P) = \dfrac{(N + 2) a + (N - 2) b}{N} \Norm{P}^2.
\end{equation}
By substituting $N = 2 m_{1} + 2 m_{2} + 2$ (see Lemma~\ref{Lem:length-N}) for the above, we complete the proof easily. 
\end{proof}

%
%
\section{Main Theorem}\label{sect:Main Theorem}

Recall that $f_{a,b}$ is the weighted squared-norm of the moment map $\mu$.
In this section, we give a proof of our main theorem, that is, $f_{a,b}$ is a Cartan-M\"unzner polynomial for some 
$a$ and $b$. 
We need $(m_1, m_2)$ defined in (\ref{eq:multiplicities}). 

\begin{Thm}
Let $G/K$ be a compact irreducible Hermitian symmetric space of rank two. 
If $a = -8(m_{1} + 2 m_{2})$ and $b = 8(2 m_{1} + m_{2} + 3)$, then $f_{a, b}(P)$ is a Cartan-M\"unzner polynomial of degree four. 
In particular, this Cartan-M\"unzner polynomial defines homogeneous isoparametric hypersurfaces in the sphere with four distinct principal curvatures with multiplicities $(m_1, m_2)$. 
\end{Thm}

\begin{proof}
In order to prove that $f_{a, b}(P)$ is a Cartan-M\"unzner polynomial of degree four, from \eqref{eq:Munzner}, we have only to show that 
\begin{align}
\Norm{\grad f_{a, b}(P)}^2 &= 16 \Norm{P}^6,\label{eq:5.A} \\ 
\Delta f_{a, b}(P) &= 8(m_{1} - m_{2})\Norm{P}^2\label{eq:5.B}. 
\end{align}
We know the formulas for $\Norm{\grad f_{a, b}(P)}^2$ and $\Delta f_{a, b}(P)$ from Propositions \ref{Prop:4.8} and \ref{Prop:4.9}. 
Let us consider their coefficients, and substitute our $a$, $b$ and 
\begin{equation}\label{eq:6.1}
a - b = - 24 (m_{1} + m_{2} + 1) 
\end{equation}
into them. We then obtain 
\begin{align}
-\dfrac{2 b^2}{(2 m_{1} + m_{2} + 3)^2} + \dfrac{(a - b)^2}{4(m_{1} + m_{2} + 1)^2} 
&= - 128 + 144 = 
16 \label{eq:A},\\[5pt]
\dfrac{12 b^2}{2 m_{1} + m_{2} + 3} + \dfrac{4 b (a - b)}{m_{1} + m_{2} + 1} 
&= 4b (24-24) =  
0\label{eq:B},\\[5pt] 
\dfrac{(m_{1} + m_{2} + 2)a + (m_{1} + m_{2})b}{m_{1} + m_{2} + 1} 
&= \dfrac{8 (m_{1}^2 - m_{2}^2 + m_{1} - m_{2})}{m_{1} + m_{2} + 1} = 
8 (m_{1} - m_{2}) \label{eq:C}. 
\end{align}
This completes the proof of the theorem. 
\end{proof}

We here list $(m_1 , m_2)$ of compact irreducible Hermitian symmetric spaces of rank two. 
One can see that $(m_1 , m_2)$ are compatible with the multiplicities of the corresponding homogeneous hypersurfaces in spheres. 

\begin{figure}[h]
\begin{tabular}{|c|c|c|}
\hline 
$G / K$ & Type & $(m_1, m_2)$\\ \hline\hline
$\SO(2 + n) / \SO(2) \times \SO(n)$ & $C_2$ &  $(1, n - 2)$\\ \hline 
$\SU(2 + n) / \mathrm{S}(\U(2) \times \U(n))$ & $BC_2$ &  $(2, 2 n - 3)$\\ \hline 
$\SO(10) / \U(5)$ & $BC_2$ &  $(4, 5)$\\ \hline 
$\mathrm{E}_6 / \U(1) \times \Spin(10)$ & $BC_2$ &  $(6, 9)$\\ \hline 
\end{tabular}
\caption{Compact irreducible Hermitian symmetric spaces of rank two}
\end{figure}

\section*{Acknowledgements}

The authors are grateful to Yoshio Agaoka, Osamu Ikawa, Reiko Miyaoka and Yoshihiro Ohnita for valuable discussions and comments. 
The authors would like to thank also, Kazuhiro Shibuya, Takahiro Hashinaga and Akira Kubo for their warm encouragements. 
The second author was supported in part by KAKENHI (20740040 and 24654012). 

%
%

\end{document}